\documentclass[12pt]{amsart}
\usepackage{latexsym, url}
\usepackage{amsthm}
\usepackage{amsmath}
\usepackage{amsfonts}
\usepackage{amssymb}
\usepackage[dvips]{graphicx}
\usepackage[hidelinks]{hyperref}
\usepackage{xypic}
\usepackage{enumitem}
\usepackage{color}
\input xy
\xyoption{all}

\newtheoremstyle{repeat}{}{}{\itshape}{}{\bfseries}{.}{.5em}{#3, repeated}

\newtheorem{theo}{Theorem}[section]
\newtheorem{lemma}[theo]{Lemma}

\newtheorem{coro}[theo]{Corollary}
\newtheorem{fact}[theo]{Fact}
\newtheorem{claim}{Claim}[theo]

\theoremstyle{definition}
\newtheorem{question}[theo]{Question}
\newtheorem{exam}[theo]{Example}
\newtheorem{exams}[theo]{Examples}

\newtheorem{rem}[theo]{Remark}
\newtheorem{defi}[theo]{Definition}
\newtheorem{conv}[theo]{Convention}
\newtheorem{constr}[theo]{Construction}

\theoremstyle{repeat}
\newtheorem*{repeated-theorem}{Repeat}

\newcommand\Ind{\operatorname{Ind}}

\newcommand\Po{\operatorname{Po}}
\newcommand\Tc{\operatorname{Tc}}
\newcommand\Rt{\operatorname{Rt}}

\newcommand\Set{\operatorname{\bf Set}}

\newcommand\mono{\operatorname{mono}}

\newcommand\colim{\operatorname{colim}}

\newcommand\ca{\mathcal {A}}

\newcommand\cc{\mathcal {C}}

\newcommand\ck{\mathcal {K}}
\newcommand\cl{\mathcal {L}}
\newcommand\cm{\mathcal {M}}

\newcommand\cx{\mathcal {X}}

\newcommand\N{\mathbb{N}}

\newcommand{\pure}{{\textup{pure}}}
\newcommand{\pbtxt}{{\textup{p.b.}}}
\renewcommand{\L}{\mathcal{L}}
\renewcommand{\phi}{\varphi}
\newcommand{\gen}[1]{\langle #1 \rangle}

\newcommand{\pushout}[1][dl]{\save*!/#1-1.5pc/#1:(-1,1)@^{|-}\restore}

\def\Ind#1#2{#1\setbox0=\hbox{$#1x$}\kern\wd0\hbox to 0pt{\hss$#1\mid$\hss}
\lower.9\ht0\hbox to 0pt{\hss$#1\smile$\hss}\kern\wd0}
\def\ind{\mathop{\mathpalette\Ind{}}}
\def\Notind#1#2{#1\setbox0=\hbox{$#1x$}\kern\wd0\hbox to 0pt{\mathchardef
\nn="3236\hss$#1\nn$\kern1.4\wd0\hss}\hbox to 0pt{\hss$#1\mid$\hss}\lower.9\ht0
\hbox to 0pt{\hss$#1\smile$\hss}\kern\wd0}

\newcommand{\nf}{\ind}

\title{Cofibrant generation of pure monomorphisms in presheaf categories}
\author{S. Cox, J. Feigert, M. Kamsma, M. Mazari-Armida and J. Rosick\'{y}}
\date{\today}
\thanks{
\begin{minipage}{0.6\textwidth}
The first author is supported by NSF grant DMS-2154141.
The third author is supported by Marie Sk\l{}odowska-Curie grant number 101130801.
The fourth author is supported by an NSF grant DMS-2348881 and a Simons Foundation grant MPS-TSM-00007597.
\end{minipage}%
\begin{minipage}{0.4\textwidth}
\begin{center}
    \includegraphics[height=1.1cm]{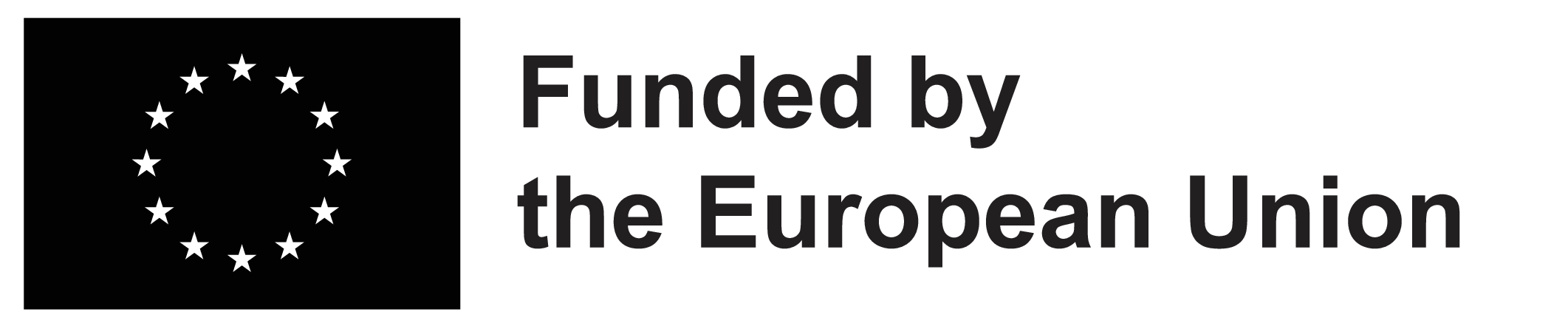}
\end{center}
\end{minipage}
\nopunct
}

\address{\newline Sean Cox\newline
Department of Mathematics and Applied Mathematics\newline
Virginia Commonwealth University\newline
1015 Floyd Ave, VA 23284, Richmond, USA}
\email{scox9@vcu.edu}
\address{\newline Jonathan Feigert\newline
Department of Mathematics\newline
Baylor University\newline
Waco, Texas, USA}
\email{Jonathan\_Feigert1@baylor.edu}
\address{\newline Mark Kamsma\newline
Department of Mathematics and Statistics\newline
Masaryk University, Faculty of Sciences\newline
Kotl\'{a}\v{r}sk\'{a} 2, 611 37, Brno, Czech Republic}
\urladdr{https://markkamsma.nl}
\email{mark@markkamsma.nl}
\address{\newline Marcos Mazari-Armida\newline
Department of Mathematics\newline
Baylor University\newline
Waco, Texas, USA}
\urladdr{https://sites.baylor.edu/marcos\_mazari/}
\email{marcos\_mazari@baylor.edu}
\address{\newline Ji\v{r}\'{i} Rosick\'{y}\newline
Department of Mathematics and Statistics\newline
Masaryk University, Faculty of Sciences\newline
Kotl\'{a}\v{r}sk\'{a} 2, 611 37, Brno, Czech Republic}
\email{rosicky@math.muni.cz}
\begin{document}

\begin{abstract}
We characterise when the pure monomorphisms in a presheaf category $\Set^\cc$ are cofibrantly generated in terms of the category $\cc$. In particular, when $\cc$ is a monoid $S$ this characterises cofibrant generation of pure monomorphisms between sets with an $S$-action in terms of $S$: this happens if and only if for all $a, b \in S$ there is $c \in S$ such that $a = cb$ or $ca = b$. We give a model-theoretic proof: we prove that our characterisation is equivalent to having a stable independence relation, which in turn is equivalent to cofibrant generation. As a corollary, we show that pure monomorphisms in acts over the multiplicative monoid of natural numbers are not cofibrantly generated.
\end{abstract}

\maketitle
\tableofcontents

\section{Introduction}
Pure monomorphisms lie between embeddings and elementary embeddings. They were introduced more than a hundred years ago  by Pr\"{ufer} for abelian groups \cite{Pr}. Since then they have been thoroughly studied from different perspectives: module theory \cite{P}, model theory \cite{R}, acts theory \cite[Chapter III]{KKM}, universal algebra \cite{W2} and category theory \cite[Section 2.D]{AR}. 

 Any module category has enough pure injectives, i.e., every module is a pure submodule of a pure injective module (for instance, of a double character module). Moreover, there is a Baer type criterion of pure injectivity: there is a test set for pure injectivity (see \cite[Lemma 1.3]{ST} or \cite[Corollary 4.10]{M-AR}). A stronger result, that pure monomorphisms are cofibrantly generated, was proved recently in \cite{LPRV}. Beyond module categories, determining if a category has enough pure injectives has been thoroughly studied \cite{BN}, \cite{T2}, \cite{T3}, see also the summary given in \cite{W2}. Already in the seventies it was known that many categories do not have enough pure injectives \cite{T3}.

Our aim is to study cofibrant generation of pure monomorphisms, which implies having enough pure injectives, in many-sorted unary universal algebras. That is, in categories $\Set^\cc$ of set-valued functors on a small category $\cc$. If $\cc$ has a single object---so if it is a monoid $S$---we get single sorted unary algebras, which are called $S$-acts. These $S$-acts have been systematically studied \cite{KKM} and can be viewed as ``non-additive modules''. The early results about the existence of enough pure injectives are due to Wenzel \cite{W} and Banaschewski \cite{Ba}. They showed that acts over the additive monoid of natural numbers and acts over groups have enough pure injectives. Recently these results were strengthened to cofibrant generation \cite{BRo} (see Remark \ref{strengthening-borceux-rosicky}). On the other hand, in \cite[Remark 5]{Ba}, it was shown that one does not necessarily have enough pure injectives in all categories of acts, and so pure monomorphisms are not cofibrantly generated in all such categories.

Sufficient conditions for $\Set^\cc$ to have enough pure injectives are given in \cite[Corollary 2.6]{BRo}. In fact, they even show that pure monomorphisms are cofibrantly generated (see the proof of \cite[Theorem 2.4]{BRo}). These conditions concern both $\cc$ and $\Set^\cc$. The condition on $\cc$ is what we call \emph{locally linearly preordered} (see Definition \ref{locally-linearly-preordered}). The condition concerning $\Set^\cc$ is that finitely presentable objects are closed under subobjects and quotients. Our work shows that the second condition is superfluous, and it establishes the converse statement.

Our main result is of algebraic and category-theoretic nature, with the notion of cofibrant generation originating from homotopy theory \cite{B}. However, our proof is model-theoretic\footnote{Some of the results were spawned by set-theoretic elementary submodel arguments, though ultimately such arguments were replaced by more powerful and precise model-theoretic ones.}. The main model-theoretic notion we use is that of a stable independence relation (see Section \ref{sec:stable-independence}), which originates from the central model-theoretic notion of Shelah's forking. This makes the results of this paper one of the first applications of the categorical approach to model-theoretic independence.

The key ingredient connecting model theory with algebra is \cite[Theorem 3.1]{LRV1}, which equates cofibrant generation to model-theoretic stability. Our proof strategy is then to show that $\Set^\cc$ is model-theoretically stable, in the sense that it admits a stable independence relation, if and only if $\cc$ is locally linearly preordered. For this, we heavily make use of the techniques of Mustafin in \cite{Mu}, where it is proved that the first-order theory of every $S$-act is stable if and only if $S$ is locally linearly preordered.

Moreover, we prove that pure monomorphisms in $\Set^\cc$ are cofibrantly generated if and only if pullback squares of pure monomorphisms are pure effective. This concept was introduced in \cite{BRo} under the name ``effective unions of pure subobjects''. It was motivated by Barr's effective unions \cite{Barr}, which are closely related to the existence of enough injectives. Pure effective unions distinguish acts from modules---abelian groups do not have effective unions of pure subobjects \cite[Example 2.5]{LPRV}.

Our main result implies that pure monomorphims in $\Set^\cc$ are cofibrantly generated if $\cc$ is a groupoid. Hence $\Set^\cc$ has enough pure injectives in this case, which extends the old result of \cite{Ba}. On the other hand, our result quickly implies that pure monomorphisms in acts over the multiplicative monoid of natural numbers are not cofibrantly generated.

\textbf{Main results.} Our main result characterises exactly when the pure monomorphisms in a presheaf category $\Set^\cc$ are cofibrantly generated, in terms of $\cc$. The characterising property is very simple.
\begin{defi}
\label{locally-linearly-preordered}
We call a category $\cc$ \emph{locally linearly preordered} if for any span of arrows $Y \xleftarrow{f} X \xrightarrow{g} Z$ in $\cc$ there is $h: Y \to Z$ such that $hf = g$ or there is $h': Z \to Y$ such that $f = h'g$.
\end{defi}
Our terminology is explained as follows. Any category induces a preorder on its collection of objects, by setting $X \leq Y$ whenever there is an arrow $X \to Y$. So Definition \ref{locally-linearly-preordered} says that for each object $X$ of $\cc$ the induced preorder on the coslice category $X/\cc$ is a linear preorder.
\begin{theo}
\label{main-theorem-simplified}
The following are equivalent for a small category $\cc$.
\begin{enumerate}
\item The pure monomorphisms in $\Set^\cc$ are cofibrantly generated.
\item There is a stable independence relation on $\Set^\cc_\pure$.
\item The category $\cc$ is locally linearly preordered.
\end{enumerate}
\end{theo}
The above is a simplified version of our main result: Theorem \ref{cofibrantly-generated-classification}, which gives further equivalent conditions that provide more information about the model-theoretic stability.

The usual strategy to obtain enough injectives from cofibrant generation applies.
\begin{coro}
\label{enough-pure-injectives}
If $\cc$ is a locally linearly preordered category then $\Set^\cc$ has enough pure injectives.
\end{coro}
To obtain the statement from the abstract concerning $S$-acts for a monoid $S$, we refer to Example \ref{monoid-as-category}.

\textbf{Overview and proof strategy.} We summarise our proof strategy, which at the same time provides an overview of this paper. As mentioned before, the aim is to show that $\Set^\cc$ admits a stable independence relation if and only if $\cc$ is locally linearly preordered. In Section \ref{sec:bad-configurations} we essentially prove that that stability implies $\cc$ being locally linearly preordered. We use the common model-theoretic theme of identifying bad infinite combinatorial configurations, and then showing that such configurations cannot occur when we have a stable independence relation (Theorem \ref{stability-implies-no-complete-bipartite-graph}). The first bad configuration amounts to interpreting an infinite complete bipartite graph (Definition \ref{complete-bipartite-fg-graph}), which can be obtained from the second bad configuration: interpreting an infinite linear order (Definition \ref{fg-order-property}), also known as the \emph{order property} in model theory. We then link this to $\cc$ being locally linearly preordered in Theorem \ref{no-order-property-implies-llp}: in Construction \ref{order-property-construction} we show how a witness of $\cc$ \emph{not} being locally linearly preordered can be turned into an instance of the order property. In Section \ref{sec:characterising-pure-effective-squares} we show that if $\cc$ is a locally linearly preordered category then the pure effective squares are precisely the pullbacks of pure monomorphisms (Theorem \ref{characterisation-of-pure-effective-squares-if-llp}). This allows us to quickly conclude that the pure effective squares form a stable independence relation (Corollary \ref{pure-effective-squares-are-stable-if-llp}). Finally, in Section \ref{sec:classification-for-cofibrant-generation} we put everything together to obtain the main result and its corollary mentioned above. We also discuss several applications of the main theorem.

\textbf{Acknowledgements.} We thank the anonymous referee for their suggestions that helped improve the presentation of this paper.

\section{Preliminaries}
We briefly introduce the notions from category theory and model theory that we will use in this paper. Further details can be found in \cite{AR, LRV, LRV1}. We will often cite results from these sources.

The main categorical context is that of locally presentable and accessible categories.
\begin{defi}
\label{accessible-category}
Let $\lambda$ be a regular cardinal. A category $\ck$ is called $\lambda$-accessible if:
\begin{itemize}
\item it has $\lambda$-directed colimits,
\item there is a set $\ca$ of $\lambda$-presentable objects such that every object is a $\lambda$-directed colimit of objects from $\ca$.
\end{itemize}
We call a category \emph{accessible} if it is $\lambda$-accessible for some $\lambda$.

A ($\lambda$-)accessible category is called \emph{locally ($\lambda$-)presentable} if it is also complete and cocomplete. For $\lambda = \omega$ we speak of \emph{finitely accessible} and \emph{locally finitely presentable} categories instead.
\end{defi}
We also briefly recall that an object $A$ of $\ck$ being \emph{$\lambda$-presentable} means that the functor $\hom(A, -): \ck \to \Set$ preserves $\lambda$-directed colimits. Intuitively, this means that $A$ ``has size $< \lambda$'', see for example the characterisation of $\lambda$-presentable presheaves at the end of Section \ref{subsec:presheaves}.
\subsection{Monoids as presheaves, and presheaves as unary algebras}
\label{subsec:presheaves}
We will work in a presheaf category $\Set^\cc$, where $\cc$ is any small category. The objects in this category are functors $\cc \to \Set$ and the arrows are natural transformations between such functors. It will be more convenient to think about $\Set^\cc$ as the category of a multi-sorted unary algebras, see below. Throughout, we fix a small category $\cc$.
\begin{defi}
\label{presheaf-as-multisorted-unary-algebra}
We define the multi-sorted signature $\cl_\cc$ as follows. It has a sort $X$ for every object $X$ in $\cc$ and it has a function symbol $f: X \to Y$ for every arrow $f: X \to Y$ in $\cc$. The (covariant) presheaves on $\cc$ are then viewed as a variety of algebras (in the sense of universal algebra) in $\cl_\cc$. That is, they are axiomatised by the following equations:
\begin{enumerate}[label=(\roman*)]
\item for each object $X$ in $\cc$, we have the equation $Id_X(x) = x$, where $Id_X$ is the identity arrow on $X$ and $x$ is a variable of sort $X$;
\item for any two composable arrows $f: X \to Y$ and $g: Y \to Z$ in $\cc$, we write $h = gf$ and we have the equation $g(f(x)) = h(x)$, where $x$ is a variable of sort $X$.
\end{enumerate}
One straightforwardly checks that homomorphisms between such algebras $K$ and $L$ are the same thing as natural transformations between $K$ and $L$ viewed as functors, which completes our algebraic perspective of $\Set^\cc$.
\end{defi}
Taking the perspective of multi-sorted unary algebras is very close to the original case where we consider (left) $S$-acts for a monoid $S$. We set up our notation in such a way that working in the greater generality of presheaves comes at no extra cost compared to working with just $S$-acts. In fact, if the reader wishes, they are invited to read everything as if the category $\cc$ is just a monoid (in the sense below).
\begin{exam}
\label{monoid-as-category}
Let $S$ be a monoid. We view $S$ as a category in the usual way: it has one object $*$ and an arrow $a: * \to *$ for every $a \in S$. The composition operation in this category is then given by the multiplication operation from the monoid. So a functor $F: S \to \Set$ is the same thing as a set $F(*)$ equipped with a left $S$-action, and natural transformations between such functors are exactly homomorphisms of left $S$-acts.

Under this identification of monoids with one-object categories we have that $S$ is locally linearly preordered if for any $a, b \in S$ there is $c \in S$ such that $a = cb$ or $ca = b$.
\end{exam}
\begin{conv}
\label{presheaf-notation-convention}
Let $K$ be a presheaf on $\cc$, by which we will always mean a covariant presheaf, viewed as a multi-sorted unary algebra (see Definition \ref{presheaf-as-multisorted-unary-algebra}).
\begin{enumerate}[label=(\roman*)]
\item Let $f: X \to Y$ be an arrow in $\cc$ and let $x$ be an element of sort $X$ in $K$. We write $f \cdot x$ for the (left) action of $f$ on $x$.
\item We use capital letters $X, Y, Z, \ldots$ for objects in $\cc$ and lowercase letters $f, g, h, \ldots$ for arrows in $\cc$. We will simply write $X \in \cc$ (resp.\ $f \in \cc$) to indicate the $X$ (resp.\ $f$) is an object (resp.\ arrow) in $\cc$.
\item We often omit the sorts (i.e., objects in $\cc$) from notation. So $x \in K$ means that $x$ is an element of some sort $X$ in $K$. Similarly, for $x \in K$ and $f \in \cc$ we will write $f \cdot x$ for the action of $f$ on $x$, which is then assumed to be of the right type (i.e., the domain of $f$ is the same as the sort of $x$).
\item Any set-theoretic operations on presheaves are computed on every sort. For example, given another presheaf $L$, we write $K \cap L$ for the presheaf whose sort $X$ is the intersection of the sort $X$ of $K$ and the sort $X$ of $L$. For $A \subseteq K$ a subset of $K$, we write $|A|$ for the sum of the cardinalities of the sorts in $A$.
\end{enumerate}
\end{conv}
Whenever we talk about an $S$-act for a monoid $S$, we will mean a \emph{left} $S$-action, matching the above convention of only considering covariant presheaves.

The category $\Set^\cc$ has many nice properties, which we briefly discuss here and will use implicitly throughout the paper. Firstly, the monomorphisms are precisely those homomorphisms which are injective on every sort. It is a locally finitely presentable category \cite[Corollary 3.7]{AR}. So in particular it has pullbacks and pushouts, and the pullback (resp.\ pushout) of a cospan (resp.\ span) of monomorphisms consists again of monomorphisms.

It will be useful to have an explicit description of such pullbacks and pushouts. A commuting square of monomorphisms
\[
\xymatrix@=1pc{
  A \ar[r] & L \\
  K \ar[u] \ar[r] & B \ar[u]
}
\]
is a pullback precisely when $K = A \cap B$. Here we view all monomorphisms as inclusions and we compute the intersection on each sort separately, as per Convention \ref{presheaf-notation-convention}(iii).

For the description of pushouts, we let $L \supseteq K \subseteq M$ be a span of inclusions of presheaves. Their pushout is simply the disjoint union over $K$. That is, the pushout is given by
\[
(L \sqcup M) / \sim,
\]
where $\sim$ identifies the copies of $K$ in $L$ and $M$. In particular, any pushout square that consists of monomorphisms is also a pullback square.

Finally, for a regular cardinal $\lambda$ bigger than the sum of the cardinalities of the sets of objects and arrows in $\cc$, a presheaf $K$ is $\lambda$-presentable if and only if $|K| < \lambda$ \cite[Example 1.14(5)]{AR}.
\begin{defi}
\label{generated-subpresheaf}
Let $K$ be a presheaf and let $A \subseteq K$ be any subset. We write $\gen{A}$ for the subpresheaf generated by $A$. That is,
\[
\gen{A} = \{ f \cdot a : a \in A, f \in \cc \}.
\]
\end{defi}

\subsection{Pure monomorphisms}
We give a logical description of the pure monomorphisms. This is the only definition that we will use. Still, for completeness' sake we also give the purely categorical characterisation (Remark \ref{categorical-pure-monomorphism}).
\begin{defi}
\label{pp-formula}
Let $\L$ be a signature. Then a \emph{positive primitive formula} or \emph{pp-formula} is one that is built from atomic formulas, (finite) conjunctions and (finite lists of) existential quantifiers.
\end{defi}
Up to logical equivalence, every pp-formula $\varphi(x_1, \ldots, x_n)$ can be written as
\[
\exists y_1 \ldots y_k \psi(x_1, \ldots, x_n, y_1, \ldots, y_k),
\]
where $\psi(x_1, \ldots, x_n, y_1, \ldots, y_k)$ is a conjunction of atomic formulas. In particular, if $\L$ consists only of constant symbols and functions symbols (as is the case when axiomatising a category of algebras, so in particular for presheaves), then $\psi$ is simply a system of equations in variables $x_1, \ldots, x_n, y_1, \ldots, y_k$.

One quickly sees that truth of pp-formulas is preserved upwards by homomorphisms. That is, if $\varphi(x_1, \ldots, x_n)$ is a pp-formula and $f: K \to L$ is a homomorphism of $\L$-structures then for all $a_1, \ldots, a_n \in K$ we have
\[
K \models \varphi(a_1, \ldots, a_n) \implies L \models \varphi(f(a_1), \ldots, f(a_n)).
\]
Being a pure monomorphism is about the converse of the above implication holding.
\begin{defi}
\label{pure-monomorphism}
A homomorphism $f: K \to L$ of presheaves on $\cc$ is called a \emph{pure monomorphism} if for every pp-formula $\varphi(x_1, \ldots, x_n)$ in the language $\L_\cc$ and all $a_1, \ldots, a_n \in K$ we have
\[
K \models \varphi(a_1, \ldots, a_n) \quad \Longleftrightarrow \quad L \models \varphi(f(a_1), \ldots, f(a_n)).
\]
\end{defi}
The name ``monomorphism'' is justified because such $f$ are in particular injective. Just apply the definition to the pp-formula $x_1 = x_2$.
\begin{rem}
\label{categorical-pure-monomorphism}
For completeness' sake we also recall the categorical characterisation of pure monomorphisms \cite[Definition 2.27]{AR}, which makes sense in any category. The equivalence between the two definitions can be seen using the proof of \cite[Proposition 5.15]{AR}, which works for any variety (i.e., category of algebras satisfying a fixed list of equations), such as the category $\Set^\cc$ of presheaves we are interested in.

An arrow $f: A \to B$ is called \emph{pure} if for each commutative square
\[
\xymatrix@=1pc{
  A \ar[r]^f & B \\
  A' \ar[r]_{f'} \ar[u]^u & B' \ar[u]_v
}
\]
where $A'$ and $B'$ are finitely presentable, the arrow $u$ factorises as $u = u'f'$ for some $u': B' \to A$.
\end{rem}
Pure monomorphisms are clearly closed under composition. In particular, in any category we can restrict to the subcategory whose morphisms are pure monomorphisms. Pure monomorphisms in a locally finitely presentable category, such as $\Set^\cc$, are always regular \cite[Proposition 2.31]{AR}, though we will not need this. We list some further facts that will be useful to us. The first two quickly follow from the definition and we provide references for the other two.
\begin{fact}
\label{pure-mono-facts}
The following general facts hold for pure monomorphisms in any locally finitely presentable category. In particular, they hold in our category of interest $\Set^\cc$.
\begin{enumerate}[label=(\roman*)]
\item Any split monomorphism is pure.
\item Pure monomorphisms are left-cancellable: if $gf$ is a pure monomorphism then so is $f$.
\item Pure monomorphisms are stable under pushout \cite[Proposition 15]{AR-pure}. That is, if $f: K \to L$ is a pure monomorphism, then its pushout along any arrow $g: K \to M$ is a pure monomorphism.
\item Pure monomorphisms are continuous \cite[Theorem 2.34]{AR}. That is, all the coprojections of a colimit of a directed diagram of pure monomorphisms are pure monomorphisms. For such a colimit and any other cocone consisting of pure monomorphisms, the induced arrow is a pure monomorphism.

In particular, pure monomorphisms are closed under transfinite composition. That is, for any ordinal $\alpha$, given a smooth chain $(f_{ij}: K_i \to K_j)_{i \leq j < \alpha}$ of pure monomorphisms (smooth means that at limit stages we have colimits), the coprojection $f_0: K_0 \to K$ is a pure monomorphism, where $K = \colim_{i < \alpha} K_i$.
\end{enumerate}
\end{fact}

\subsection{Stable independence and cofibrant generation}
\label{sec:stable-independence}
The main model-theoretic tool in this paper is that of a stable independence relation. The categorical approach to independence was initiated in \cite{LRV}, and we use the same terminology. We only recall the definitions that are important to us, and refer to \cite{LRV} for the rest.
\begin{defi}
\label{weakly-stable-independence}
An independence relation $\ind$ on a category $\ck$ is a collection of commuting squares (invariant under isomorphic squares). The squares in the relation are called \emph{independent squares}. We often leave the arrows involved unlabelled and we write $A \ind_C^D B$ if the square below is independent. Furthermore, we require that in the commuting diagram below we have $A \nf_C^D B$ if and only if $A \ind_C^E B$.
\[
\xymatrix@=1pc{
  A \ar[r] & D \ar[r] & E \\
  C \ar[r] \ar[u] & B \ar[u] &
}
\]
We say that $\ind$ is \emph{weakly stable} if it satisfies symmetry, uniqueness, transitivity and existence \cite[Definitions 3.9, 3.13, 3.15 and 3.10]{LRV}.
\end{defi}
It follows from existence that any square where the bottom or left arrow is an isomorphism is independent \cite[Lemma 3.12]{LRV}. The transitivity property says exactly that when composing two independent squares, the outer rectangle is again independent \cite[Definition 3.15]{LRV}. With these properties, the independence relation itself becomes a category.
\begin{defi}[{\cite[Definition 3.16]{LRV}}]
\label{independence-category}
Given a category $\ck$ we write $\ck^2$ for the category that has as objects arrows from $\ck$ and as arrows the commuting squares in $\ck$. For an independence relation $\ind$ on $\ck$ that satisfies existence and transitivity we then define $\ck_{\ind}$ to be the subcategory of $\ck^2$ where the arrows are independent squares.
\end{defi}
The important property that is missing from a weakly stable independence relation is that the relation is built from a small amount of information. This can be made precise using the language of accessible categories.
\begin{defi}[{\cite[Definition 3.24]{LRV}}]
\label{stable-independence}
Let $\ck$ be a category, equipped with an independence relation $\ind$ that satisfies existence and transitivity. We call $\ind$ \emph{accessible} if $\ck_{\ind}$ is an accessible category. We call $\ind$ \emph{stable} if it is both weakly stable and accessible.
\end{defi}
Given a category $\ck$ and a class of arrows $\cm$ (closed under composition and containing identity arrows) we will write $\ck_\cm$ for the wide subcategory of $\ck$ with arrows from $\cm$. That is, its objects are the same as those in $\ck$ and its arrows are exactly those in $\cm$. For these data, \cite[Section 2]{LRV1} gives us a recipe to construct a weakly stable independence relation using effective squares (called ``cellular squares'' in \cite{LRV1}). We will only be interested in the case where $\cm$ is the class of pure monomorphisms, so we restrict ourselves to that case.
\begin{defi}
\label{pure-effective}
Let $\ck$ be a category with pushouts. A commuting square of pure monomorphisms in $\ck$, such as the outer square below, is called \emph{pure effective} if the induced arrow from the relevant pushout is a pure monomorphism. That is, if the arrow $P \to D$ is a pure monomorphism, where $P$ is the pushout.
\[
\xymatrix@=1.8pc{
 & & D \\
A \ar[r] \ar@/^/[urr] & P \ar@{-->}[ur] \pushout & \\
C \ar[r] \ar[u] & B \ar[u] \ar@/_/[uur] &
}
\]
\end{defi}
\begin{fact}
\label{pure-effective-weakly-stable}
Let $\ck$ be a locally finitely presentable category. Then the pure effective squares form a weakly stable independence relation on $\ck_\pure$.
\end{fact}
\begin{proof}
This is \cite[Theorem 2.7]{LRV1}. Fact \ref{pure-mono-facts} tells us that the conditions for that theorem are satisfied (i.e., $(\ck, \pure)$ is a cellular category).
\end{proof}
We finish this section with the link between stable independence and cofibrant generation. The essence is that the former is all about the independence relation being built from a small amount of information, while the latter means that a fixed class of arrows is generated from a small amount of information.
\begin{defi}[{\cite[Remark 1.2]{B}}]
\label{cofibrant-generation}
Let $\cm$ be a class of arrows in a category $\ck$.
\begin{itemize}
\item We write $\Po(\cm)$ for the closure of $\cm$ under pushouts (in $\ck$).
\item We write $\Tc(\cm)$ for the closure of $\cm$ under transfinite composition (in $\ck$).
\item We write $\Rt(\cm)$ for the closure of $\cm$ under retracts (in $\ck^2$).
\end{itemize}
We call $\cm$ \emph{cofibrantly generated} if there is a set $\cx$ of arrows such that $\cm = \Rt(\Tc(\Po(\cx)))$.
\end{defi}
We can now state the equivalence from \cite[Theorem 3.1]{LRV1} between cofibrant generation and admitting a stable independence relation. We will simplify the statement for our case of interest: pure monomorphisms in a locally finitely presentable category.
\begin{fact}
\label{cofibrant-generation-iff-stable}
Let $\ck$ be a locally finitely presentable category. Then the following are equivalent:
\begin{enumerate}[label=(\roman*)]
\item $\ck_\pure$ has a stable independence relation,
\item the pure effective squares form a stable independence relation on $\ck_\pure$,
\item the pure monomorphisms are cofibrantly generated in $\ck$.
\end{enumerate}
\end{fact}
\begin{proof}
The fact that all assumptions for \cite[Theorem 3.1]{LRV1} are satisfied is Fact \ref{pure-mono-facts}, where closure under retracts quickly follows from (i) and (ii) in Fact \ref{pure-mono-facts}.
\end{proof}

\section{Bad configurations}
\label{sec:bad-configurations}
A common theme in model theory is to identify a certain bad combinatorial configuration, and declare that we do not wish that to happen. We will identify such configurations for presheaves, and these will come up as intermediate steps in the proof of our main theorem. Such bad configurations are of infinite size. Importantly, using the compactness theorem, we can then find such configurations of arbitrarily big sizes. So we begin by recalling the compactness theorem.
\begin{fact}[{Compactness theorem, e.g., \cite[Theorem 2.2.1]{TZ}}]
\label{compactness}
Let $\L$ be a signature and let $\Sigma$ be a set of $\L$-sentences. If for every finite $\Sigma_0 \subseteq \Sigma$ there is an $\L$-structure $M$ such that $M \models \Sigma_0$ then there is an $\L$-structure $N$ such that $N \models \Sigma$.
\end{fact}
\begin{defi}
\label{complete-bipartite-fg-graph}
Let $f,g \in \cc$, and let $K$ be a presheaf on $\cc$. An \emph{$(f,g)$-interpreted complete bipartite graph in $K$} consists of infinite sets $A, B \subseteq K$ such that for all $a \in A$ and $b \in B$ there is $c \in K$ with $f \cdot c = a$ and $g \cdot c = b$.

We say that $K$ \emph{interprets a complete bipartite graph} if there are $f, g \in \cc$ such that $K$ contains an $(f,g)$-interpreted complete bipartite graph.
\end{defi}
The terminology in the above definition comes from the following. The formula $\exists z(f \cdot z = x \wedge g \cdot z = y)$ defines an edge relation between $x$ and $y$, which turns $(A, B)$ into the homomorphic image of a complete bipartite graph.
\begin{lemma}
\label{complete-bipartite-fg-graph-equivalences}
The following are equivalent for $f, g \in \cc$.
\begin{enumerate}[label=(\roman*)]
\item There is a presheaf $K$ containing an $(f,g)$-interpreted complete bipartite graph.
\item For every infinite cardinal $\lambda$ there is a presheaf $K$ containing a $(f,g)$-interpreted complete bipartite graph $(A, B)$ such that $|A| = |B| = \lambda$.
\end{enumerate}
\end{lemma}
\begin{proof}
\underline{(i) $\Rightarrow$ (ii)} We prove this by a standard argument using the compactness theorem (Fact \ref{compactness}), for which we give some details. Let $K$ be a presheaf containing a $(f,g)$-interpreted complete bipartite graph $(A, B)$. Extend the signature $\L_\cc$ to $\L'$ by adding new constant symbols $(a_i)_{i < \lambda}$ and $(b_i)_{i < \lambda}$. Let $T$ be the set of $\L_\cc$-sentences that axiomatises presheaves (see Definition \ref{presheaf-as-multisorted-unary-algebra}). Define the set $\Sigma'$ of $\L'$-sentences as:
\[
\{ \exists z(f \cdot z = a_i \wedge g \cdot z = b_j) : i,j < \lambda \} \cup \{a_i \neq a_{i'} : i < i' < \lambda \} \cup \{b_j \neq b_{j'} : j < j' < \lambda\},
\]
and set $\Sigma = T \cup \Sigma'$. For any finite $\Sigma_0 \subseteq \Sigma$ the new constant symbols can be interpreted in $K$ (specifically, in $A$ and $B$) such that $K \models \Sigma_0$. By the compactness theorem (Fact \ref{compactness}) there is an $\L'$-structure $N$ such that $N \models \Sigma$. Set $A' = \{a_i \in N : i < \lambda\}$ and $B' = \{b_i \in N : i < \lambda\}$, then $N$ restricted to $\L_C$ is the desired presheaf with $(A', B')$ as the $(f,g)$-interpreted complete bipartite graph.

\underline{(ii) $\Rightarrow$ (i)} Immediate.
\end{proof}
\begin{defi}
\label{fg-order-property}
Let $f,g \in \cc$, and let $K$ be a presheaf. We say that $K$ has the \emph{$(f, g)$-order property} if there are sequences $(a_i)_{i < \omega}$ and $(b_i)_{i < \omega}$ of elements of $K$ such that the following holds. There is $c \in K$ with $f \cdot c = a_i$ and $g \cdot c = b_j$ if and only if $i \leq j$.

We say that $K$ has the \emph{span-induced order property} if there are $f,g \in \cc$ such that $K$ has the $(f, g)$-order property.
\end{defi}
The definition of the $(f, g)$-order property can also be expressed as
\[
K \models \varphi_{f,g}(a_i, b_j) \quad \Longleftrightarrow \quad i \leq j < \omega,
\]
where $\varphi_{f,g}(x, y)$ is the formula $\exists z(f \cdot z = x \wedge g \cdot z = y)$. So this is exactly saying that $\varphi_{f,g}(x, y)$ has the model-theoretic order property (see e.g., \cite[Theorem II.2.2]{Shelah}).
\begin{lemma}
\label{order-property-implies-complete-bipartite-graph}
If there is a presheaf that has the span-induced order property then there is a presheaf that interprets a complete bipartite graph.
\end{lemma}
\begin{proof}
Let $f,g \in \cc$ such that $K$ has the $(f, g)$-order property. By the compactness theorem (following an argument similar to Lemma \ref{complete-bipartite-fg-graph-equivalences}(i) $\Rightarrow$ (ii)), we may assume that there are sequences $(a_i)_{i < 2\omega}$ and $(b_i)_{i < 2\omega}$ in $K$ such that there is $c \in K$ with $f \cdot c = a_i$ and $g \cdot c = b_j$ if and only if $i \leq j$. Let $A = \{ a_i : i < \omega \}$ and $B = \{ b_i : \omega \leq i < 2 \omega \}$. Then $(A, B)$ is an $(f, g)$-interpreted complete bipartite graph.
\end{proof}
As mentioned in the introduction of this section, the combinatorial properties we defined---interpreting a complete bipartite graph and having the span-induced order property---are bad things.
\begin{theo}
\label{stability-implies-no-complete-bipartite-graph}
If the pure effective squares form a stable independence relation in $\Set^\cc_\pure$ then there is no presheaf that interprets a complete bipartite graph.
\end{theo}
\begin{proof}
Let $\ind$ be the independence relation on $\Set^\cc_\pure$ given by pure effective squares. We will freely use that the inclusions $\Set^\cc_\pure \hookrightarrow \Set^\cc$ and $\Set^\cc_{\pure,\ind} \hookrightarrow (\Set^\cc)^2$ preserve directed colimits (see Fact \ref{pure-mono-facts}(iv) and \cite[Lemma 2.11]{LRV1}), and that directed colimits in $(\Set^\cc)^2$ are computed component-wise (see \cite[Exercise 2c]{AR}). Hence all directed colimits are just computed by taking unions.

Restricting to pure monorphisms in a finitely accessible category yields an accessible category, by \cite[Theorem 2.34]{AR}. So $\Set^\cc_\pure$ is accessible. For any set of accessible categories there are arbitrarily large cardinals $\lambda$ such that every category in that set is $\lambda$-accessible, by \cite[Theorem 2.19]{AR}. There are thus regular cardinals $\mu < \lambda$, with $\mu$ bigger than the sum of the cardinalities of the sets of objects and arrows in $\cc$, such that $\Set^\cc_\pure$ is $\lambda^+$-accessible and $\Set^\cc_{\pure,\ind}$ is $\mu^+$-accessible. In fact, \cite[Theorem 2.19]{AR} also applies to functors, so we may assume that the inclusion functors $\Set^\cc_\pure \hookrightarrow \Set^\cc$ and $\Set^\cc_{\pure,\ind} \hookrightarrow (\Set^\cc)^2$ preserve $\lambda^+$-presentable and $\mu^+$-presentable objects respectively.

Suppose for a contradiction that there is a presheaf that interprets a complete bipartite graph. By Lemma \ref{complete-bipartite-fg-graph-equivalences} there is a presheaf $K$ containing an $(f,g)$-interpreted complete bipartite graph $(A, B)$, for some $f,g \in \cc$, with $|A| = \lambda^+$ and $|B| = \lambda^+$.

In what follows, we will view all pure monomorphisms as actual inclusions maps to simplify notation. As $\Set^\cc_\pure$ is $\lambda^+$-accessible, there is a pure monomorphism $u: L \to K$, where $L$ is $\lambda^+$-presentable (in $\Set^\cc_\pure$, and hence in $\Set^\cc$) and such that $L$ contains $\lambda$ many elements from $A$.\footnote{For the reader unfamiliar with this kind of argument, we explain how this can be done. By definition of being a $\lambda^+$-accessible category we can write $K$ as a $\lambda^+$-directed colimit $K = \colim_{i \in I} K_i$ of $\lambda^+$-presentable objects in $\Set^\cc_\pure$. Directed colimits in $\Set^\cc_\pure$ are computed as in $\Set^\cc$, i.e., simply as the union. Pick $\lambda$ many elements $(a_j)_{j < \lambda}$ in $A \subseteq K$. For each $j < \lambda$ there is $i_j \in I$ such that $a_j \in K_{i_j}$. By $\lambda^+$-directedness there is $i \in I$ with $i \geq i_j$ for all $j < \lambda$. We can now take $L = K_i$.} In particular, we have $|L| = |L \cap A| = \lambda$. As $|B| = \lambda^+ > \lambda = |L|$ there must be $b \in B \setminus L$. Since $\Set^\cc_{\pure,\ind}$ is $\mu^+$-accessible, there is a pure effective square
\[
\xymatrix@=1pc{
  L \ar[r]^{u} & K \\
  L_0 \ar [u]^{v} \ar[r]_{u_0} &
  K_0 \ar[u]_{w}
}
\]
with $b \in K_0$ and $u_0: L_0 \to K_0$ is $\mu^+$-presentable in $\Set^\cc_{\pure,\ind}$. So $u_0$ is $\mu^+$-presentable in $(\Set^\cc)^2$, which means precisely that $L_0$ and $K_0$ are both $\mu^+$-presentable in $\Set^\cc$ (see \cite[Exercise 2.c]{AR}). That is, $|L_0| \leq \mu$ and $|K_0| \leq \mu$. Considering cardinalities again, there must be $a \in (L \cap A) \setminus K_0$.

Hence, in the pushout
\[
\xymatrix@=1pc{
  L \ar[r] & P \\
  L_0 \ar [u]^{v} \ar[r]_{u_0} &
  K_0 \ar[u]
}
\]
there is no $z$ such that $a, b \in \gen{z}$, because that would imply $z \in L$ or $z \in K_0$ and hence $b \in L$ or $a \in K_0$. So $P \not \models \exists z(f \cdot z = a \wedge g \cdot z = b)$ while $K \models \exists z(f \cdot z = a \wedge g \cdot z = b)$, and so the induced arrow $P \to K$ is not pure.
\end{proof}
\begin{theo}
\label{no-order-property-implies-llp}
If there is no presheaf that has the span-induced order property then $\cc$ is a locally linearly preordered category.
\end{theo}
To prove the above theorem we generalise the construction in \cite[Section 3]{Mu} from sets with a monoid action to presheaves (see Construction \ref{order-property-construction}). We first explain how this yields the above theorem, for which we recall the notion of a representable presheaf (see Definition \ref{representable-presheaf}). We then dedicate the remainder of this section to working out the details of Construction \ref{order-property-construction}.
\begin{defi}
\label{representable-presheaf}
Recall that the \emph{representable presheaf} $\cc(X, -)$ on a fixed object $X \in \cc$ is defined as follows. The elements of an arbitrary sort $Y$ are the arrows $X \to Y$ in $\cc$, that is, the set $\cc(X, Y)$. For an arrow $f: Y \to Z$ we define its action by composition. That is, we need to define a function $\cc(X, Y) \to \cc(X, Z)$. So given $a \in \cc(X, Y)$, which is an arrow $a: X \to Y$, we define $f \cdot a$ to be the composition $f \circ a: X \to Z$.
\end{defi}
If $\cc$ is a monoid $S$ (see Example \ref{monoid-as-category}) then there is only one sort. Working out the above definition for that one sort, we get $S$ itself with the obvious $S$-action.
\begin{proof}[Proof of Theorem \ref{no-order-property-implies-llp}]
We prove the contrapositive. So suppose that $\cc$ is not locally linearly preordered. Then there is a span $Y \xleftarrow{f} X \xrightarrow{g} Z$ such that there is no $h: Y \to Z$ with $hf = g$ and there is also no $h': Z \to Y$ such that $f = h'g$. Let $K = \cc(X, -)$ be the representable presheaf on $X$ (see Definition \ref{representable-presheaf}). Then Construction \ref{order-property-construction} with $a = f$, $b = g$ and $c = Id_X$ yields a presheaf $K_\infty$ together with sequences $(a_n)_{n < \omega}$ and $(b_m)_{m < \omega}$ of elements in $K_\infty$, such that the following hold.
\begin{enumerate}
\item Construction \ref{order-property-construction} includes elements $c_{n,m} \in K_\infty$ for all $m \leq n < \omega$, such that $f \cdot c_{n,m} = a_n$ and $g \cdot c_{n,m} = b_m$,
\item For all $n < m < \omega$ there is no $d \in K_\infty$ with $a_n \in \gen{d}$ and $b_m \in \gen{d}$, by Lemma \ref{order-property-construction-no-connections}. In particular, there is no $d \in K_\infty$ such that $f \cdot d = a_n$ and $g \cdot d = b_m$.
\end{enumerate}
Together, (i) and (ii) above say exactly that $K_\infty$ has the span-induced order property.
\end{proof}
\begin{constr}
\label{order-property-construction}
Let $K$ be a presheaf and let $a, b, c \in K$ be such that:
\begin{itemize}
\item there is $f \in \cc$ such that $a = f \cdot c$,
\item there is $g \in \cc$ such that $b = g \cdot c$,
\item there is no $h \in \cc$ such that $h \cdot a = b$,
\item there is no $h' \in \cc$ such that $a = h' \cdot b$.
\end{itemize}
Let $I = \{(n, m) : m \leq n < \omega\}$ and order $I$ lexicographically. Note that $I$ has the same order type as $\omega$. We construct a chain $(K_i)_{i \in I}$ of presheaves by induction on $i$, together with a sequence of elements $(c_i)_{i \in I}$, as well as the following data.
\begin{itemize}
\item The first place in the chain where $c_i$ appears is $K_i$. That is, $c_i \in K_i$, and $c_i \not \in K_j$ for all $j < i$.
\item There are elements $(a_n)_{n < \omega}$ such that $a_n = f \cdot c_{n, m}$ for all $m \leq n < \omega$. The first place in the chain where $a_n$ appears is $K_{n,0}$.
\item There are elements $(b_m)_{m < \omega}$ such that $b_m = g \cdot c_{n, m}$ for all $m \leq n < \omega$. The first place in the chain where $b_m$ appears is $K_{m,m}$.
\item For all $m \leq n < \omega$ there is a monomorphism $v_{n,m}: K \to K_{n,m}$ such that $c_{n,m} = v_{n,m}(c)$.
\end{itemize}
Having constructed the chain, we set $K_\infty = \bigcup_{i \in I} K_i$. We construct $K_{n, m}$ based on the following cases.\\
\\
\underline{The base case, $n = m = 0$.} We start the inductive construction by setting $K_{0,0} = K$ and $c_{0,0} = c$. Then $a_0 = a$ and $b_0 = b$.\\
\\
\underline{The case $n > m = 0$.} Let $u_{n,0}: \gen{b} \to K_{n-1,n-1}$ be the map given by $b \mapsto b_0$. We construct the link $K_{n-1,n-1} \to K_{n, 0}$ as the following pushout.
\[
\xymatrix@=1.8pc{
K_{n-1,n-1} \ar[r] & K_{n,0} \pushout \\
\gen{b} \ar@{^{(}->}[r] \ar[u]^{u_{n,0}} & K \ar[u]_{v_{n,0}}
}
\]
We set $c_{n,0} = v_{n,0}(c)$, and so $a_n = v_{n,0}(a)$.\\
\\
\underline{The case $n > m > 0$.} Let $u_{n,m}: \gen{a, b} \to K_{n,m-1}$ be the map given by $a \mapsto a_n$ and $b \mapsto b_m$. We construct the link $K_{n,m-1} \to K_{n, m}$ as the following pushout.
\[
\xymatrix@=1.8pc{
K_{n,m-1} \ar[r] & K_{n,m} \pushout \\
\gen{a, b} \ar@{^{(}->}[r] \ar[u]^{u_{n,m}} & K \ar[u]_{v_{n,m}}
}
\]
We set $c_{n,m} = v_{n,m}(c)$.\\
\\
\underline{The case $n = m > 0$.} Let $u_{n,m}: \gen{a} \to K_{n,m-1}$ be the map given by $a \mapsto a_n$. We construct the link $K_{n,m-1} \to K_{n, m}$ as the following pushout.
\[
\xymatrix@=1.8pc{
K_{n,m-1} \ar[r] & K_{n,m} \pushout \\
\gen{a} \ar@{^{(}->}[r] \ar[u]^{u_{n,m}} & K \ar[u]_{v_{n,m}}
}
\]
We set $c_{n,m} = v_{n,m}(c)$, and so $b_m = v_{n,m}(b)$.
\end{constr}
\begin{lemma}[Well-definedness of Construction \ref{order-property-construction}]
\label{order-property-construction-well-defined}
In Construction \ref{order-property-construction}, set $H = \gen{a} \cap \gen{b} = \gen{a_0} \cap \gen{b_0} \subseteq K = K_{0,0}$. We prove that the following properties hold at stage $(n, m) \neq (0, 0)$ by induction on the construction, which in particular shows that each step in the construction is well-defined.
\begin{enumerate}[label=(\roman*)]
\item The map $u_{n,m}$ is a well-defined monomorphism. In particular, all the maps in the construction are monomorphisms.
\item For all $n',m' < \omega$ such that $a_{n'}$ and $b_{m'}$ exist in $K_{n,m}$ we have $\gen{a_{n'}} \cap \gen{b_{m'}} = H$.
\item For all $0 \leq n' < n$ the map $\alpha_{n',n}: \gen{a_{n'}} \to \gen{a_n}, h \cdot a_{n'} \mapsto h \cdot a_n$ is injective and well-defined. Furthermore, we have $h \cdot a_n \in H$ if and only if $h \cdot a_n = h \cdot a_0$.
\item For all $0 \leq m' < m$ the map $\beta_{m',m}: \gen{b_{m'}} \to \gen{b_m}, h \cdot b_{m'} \mapsto h \cdot b_m$ is injective and well-defined. Furthermore, we have $h \cdot b_m \in H$ if and only if $h \cdot b_m = h \cdot b_0$.
\end{enumerate}
\end{lemma}
\begin{proof}
We prove each item separately.
\begin{enumerate}[label=(\roman*)]
\item We consider each of the cases in the construction. The case $n > m = 0$ is immediate, as $u_{n,0}$ is a genuine inclusion. The case $n = m > 0$ follows from (iii), as then $u_{n,m} = \alpha_{0,n}$. The case where $n > m > 0$ remains. In this case $a_n$ and $b_m$ already existed at an earlier stage (to be precise, $a_n$ was the last one to be constructed, at stage $(n,0)$). So, by the induction hypothesis, we can apply (ii)--(iv) to these elements. As $u_{n,m} = \alpha_{0,n} \cup \beta_{0,m}$ and $\alpha_{0,n}$ and $\beta_{0,m}$ are injective by (iii) and (iv), we only need to show that $h \cdot a_0 = h' \cdot b_0$ if and only if $h \cdot a_n = h' \cdot b_m$, for all $h,h' \in \cc$.

Suppose that $h \cdot a_0 = h' \cdot b_0$, then this element is in $H$ by definition. So by (iii) and (iv) we have $h \cdot a_n = h \cdot a_0 = h' \cdot b_0 = h' \cdot b_m$, as required.

Conversely, suppose that $h \cdot a_n = h' \cdot b_m$. Then by (ii) this element is in $H$, and so we can apply (iii) and (iv) again to conclude that $h \cdot a_0 = h \cdot a_n = h' \cdot b_m = h' \cdot b_0$.

We have thus shown that $u_{n,m}$ is a well-defined monomorphism for all $m \leq n < \omega$. The conclusion that all constructed maps are monomorphisms follows, because they are all constructed as pushouts of monomorphisms.

\item By (i) all the maps are monomorphisms, so $\gen{a_{n'}} \cap \gen{b_{m'}}$ is computed the same in $K_{n',m'}$ as in $K_{n,m}$. We can thus limit ourselves to the stages where the last one of $a_{n'}$ and $b_{m'}$ is constructed. We will treat the case where $a_{n'}$ is constructed last (i.e., $m' < n'$), so this is in the construction of $K_{n',0}$. The other case ($m' \geq n'$) is similar by considering the construction of $K_{m',m'}$.

As pushouts of monomorphisms are pullbacks and $\gen{a_{n'}} = \gen{v_{n',0}(a)} \subseteq v_{n',0}(K)$ and $\gen{b_{m'}} \subseteq K_{m',m'} \subseteq K_{n'-1,n'-1}$, we have $\gen{a_{n'}} \cap \gen{b_{m'}} \subseteq \gen{b} = v_{n',0}(\gen{b})$. Hence
\begin{align*}
\gen{a_{n'}} \cap \gen{b_{m'}} &= \gen{a_{n'}} \cap \gen{b_{m'}} \cap v_{n',0}(\gen{b}) \\
&= v_{n',0}(\gen{a} \cap \gen{b}) \cap \gen{b_{m'}} \\
&= v_{n',0}(H) \cap \gen{b_{m'}} \\
&= H \cap \gen{b_{m'}},
\end{align*}
using (i) to tell us that $v_{n',0}$ is a monomorphism, and $v_{n',0}$ restricted to $\gen{b}$ is the inclusion, so $v_{n',0}(H) = H$. By (ii) we have $H = \gen{a_{m'}} \cap \gen{b_{m'}} \subseteq \gen{b_{m'}}$, and the claim follows.

\item There is only something to prove at stage $(n, 0)$, because that is where $a_n$ is constructed. By (i) the map $v_{n,0}$ is a monomorphism, so because $a_n = v_{n,0}(a)$ we have that $\alpha_{0,n}$ is injective and well-defined. For $0 < n' < n$ we have that $\alpha_{n',n} = \alpha_{0,n} \alpha_{0,n'}^{-1}$ is well-defined and injective by applying (iii) from the inductive hypothesis to $n' > 0$.

Now suppose that $h \cdot a_n \in H$. Then there is $h' \in \cc$ such that $h \cdot a_n = h' \cdot b$, and so $v_{n,0}(h \cdot a) = h \cdot v_{n,0}(a) = h \cdot a_n = h' \cdot b = h' \cdot v_{n,0}(b) = v_{n,0}(h' \cdot b)$. By injectivity of $v_{n,0}$, see (i), we thus have $h \cdot a = h' \cdot b$ in $K$. As $a_0 = a$ in $K_{0,0} = K$, we conclude $h \cdot a_n = h' \cdot b = h \cdot a_0$.

Conversely, suppose that $h \cdot a_n = h \cdot a_0$. Then this element is in $K_{n-1,n-1} \cap v_{n,0}(K)$. As pushouts of monomorphisms are pullbacks, we have that $K_{n-1,n-1} \cap v_{n,0}(K) = \gen{b_0}$. So $h \cdot a_n \in \gen{b_0}$, and thus $h \cdot a_n \in \gen{a_0} \cap \gen{b_0} = H$.

\item Analogous to (iii), limiting ourselves to the stages where $n = m$ instead and using that $H = \gen{a_m} \cap \gen{b_0}$ by (ii). \qedhere
\end{enumerate}
\end{proof}
Lemma \ref{order-property-construction-well-defined} is mainly meant to show that Construction \ref{order-property-construction} is well-defined. However, the helper set $H$ in that lemma is useful, so we extract its useful properties.
\begin{coro}
\label{order-property-construction-set-H}
Working in $K_\infty$ from Construction \ref{order-property-construction}, we set $H = \gen{a_0} \cap \gen{b_0}$. Then
\begin{enumerate}[label=(\roman*)]
\item $H = \gen{a_n} \cap \gen{b_m}$ for all $n, m < \omega$,
\item $H = \gen{a_n} \cap \gen{a_{n'}} = \gen{b_n} \cap \gen{b_{n'}}$ for all $n \neq n'$ in $\omega$,
\item $a_n \not \in H$, $b_m \not \in H$ and $c_{n,m} \not \in H$ for all $m \leq n < \omega$.
\end{enumerate}
\end{coro}
\begin{proof}
We prove each item separately.
\begin{enumerate}[label=(\roman*)]
\item This is Lemma \ref{order-property-construction-well-defined}(ii).
\item We prove that $\gen{a_n} \cap \gen{a_{n'}} = H$. Without loss of generality, assume that $n' < n$. Then $\gen{a_{n'}} \subseteq K_{n-1,n-1}$ and $\gen{a_n} \subseteq v_{n,0}(K)$. As the pushout square that is used to construct $K_{n,0}$ is also a pullback square, we have that $\gen{a_{n'}} \cap \gen{a_n} \subseteq \gen{b} = \gen{b_0}$. So, using (i), we have $\gen{a_{n'}} \cap \gen{a_n} = \gen{a_{n'}} \cap \gen{a_n} \cap \gen{b_0} = (\gen{a_{n'}} \cap \gen{b_0}) \cap (\gen{a_n} \cap \gen{b_0}) = H$. The case $\gen{b_m} \cap \gen{b_{m'}} = H$ is similar.
\item We prove that $a_n \not \in H$ (and $b_m \not \in H$ is similar). Then $c_{n,m} \not \in H$ follows, as $f \cdot c_{n,m} = a_n$ and $H$ is by construction closed under the the actions by $\cc$ (i.e., $H$ is a presheaf itself).

Suppose for a contradiction that $a_n \in H$. Then by Lemma \ref{order-property-construction-well-defined}(iii) we have that $a_0 \in H$. So there is $h \in \cc$ such that $a_0 = h \cdot b_0$. This contradicts the assumptions on $a = a_0$ and $b = b_0$ in Construction \ref{order-property-construction}. \qedhere
\end{enumerate}
\end{proof}
\begin{lemma}
\label{order-property-construction-no-connections}
Working in $K_\infty$ from Construction \ref{order-property-construction} we have for all $n < m < \omega$ that there is no $d \in K_\infty$ with $a_n \in \gen{d}$ and $b_m \in \gen{d}$.
\end{lemma}
\begin{proof}
Suppose for a contradiction that there are such $n < m < \omega$ and $d \in K_\infty$. Let $m' \leq n' < \omega$ be such that $K_{n',m'}$ is the first link in the chain such that $d \in K_{n',m'}$. Let $K'$ be the preceding link in the chain, so $K_{n',m'}$ is constructed as the pushout $K' \xleftarrow{u_{n',m'}} \gen{U} \hookrightarrow K$, where $U \subseteq \{a,b\} \subseteq K$ with at least one of $a$ and $b$ in $U$. Since pushouts of monomorphisms are pullbacks, we have that $C := K' \cap v_{n',m'}(K)$ can be computed as follows
\[
C = \begin{cases}
\gen{b} & \text{if } m' = 0, \\
\gen{a_{n'}, b_{m'}} & \text{if } 0 < m' < n', \\
\gen{a_{n'}} & \text{if } m' = n'.
\end{cases}
\]
As $d \in K_{n',m'} \setminus K'$, we must have $d \in v_{n',m'}(K)$ and so $a_n, b_m \in v_{n',m'}(K)$. At the same time, we have that $b_m \not \in K_{n,0}$ because $m > n$, and so $d \not \in K_{n,0}$. As $K'$ is the last link where $d$ does not appear, we have that $K_{n,0} \subseteq K'$, and so $a_n \in K'$. We thus see that $a_n \in C$. We distinguish three cases, based on the computation of $C$, and show that $a_n \in H$, contradicting Corollary \ref{order-property-construction-set-H}(iii).
\begin{itemize}
\item The case $m' = 0$. Then $a_n \in \gen{b}$, and hence $a_n \in \gen{a_n} \cap \gen{b} = H$ by Corollary \ref{order-property-construction-set-H}(i).
\item The case $0 < m' < n'$. We distinguish two subcases, based on the fact that $C = \gen{a_{n'}, b_{m'}} = \gen{a_{n'}} \cup \gen{b_{m'}}$.
\begin{itemize}
\item The subcase $a_n \in \gen{a_{n'}}$. As $d \in K_{n',m'}$ we also have $b_m \in K_{n', m'}$, and since $b_m$ is constructed at stage $(m,m)$ we have that $(m,m) \leq (n', m')$ in $I$. In particular, $n < m \leq n'$. So by Corollary \ref{order-property-construction-set-H}(ii) we have $a_n \in \gen{a_n} \cap \gen{a_{n'}} = H$.
\item The subcase $a_n \in \gen{b_{m'}}$. Then $a_n \in \gen{a_n} \cap \gen{b_{m'}} = H$ by Corollary \ref{order-property-construction-set-H}(i).
\end{itemize}
\item The case $m' = n'$. Then $a_n \in \gen{a_{n'}}$ and the same argument as in the first subcase above applies.
\end{itemize}
We thus conclude that such $n < m$ and $d \in K_\infty$ cannot exist.
\end{proof}

\section{Characterising pure effective squares}
\label{sec:characterising-pure-effective-squares}
We start with the statement of the main theorem of this section, and its important corollary. The remainder of this section is then devoted to proving the theorem.
\begin{theo}
\label{characterisation-of-pure-effective-squares-if-llp}
If $\cc$ is a locally linearly preordered category then a commuting square in $\Set^\cc$ is pure effective if and only if it is a pullback square consisting of pure monomorphisms.
\end{theo}
\begin{coro}
\label{pure-effective-squares-are-stable-if-llp}
Suppose that $\cc$ is a locally linearly preordered category. Then the pure effective squares form a stable independence relation on $\Set^\cc_\pure$.
\end{coro}
\begin{proof}
Let $\ind$ be the independence relation on $\Set^\cc_\pure$ given by the pure effective squares. By Fact \ref{pure-effective-weakly-stable} $\ind$ is weakly stable, so we only need to show that $\Set^\cc_{\pure,\ind}$ is accessible. We could argue how to directly apply \cite[Theorem 3.4]{KR}, but it is easier to repeat the argument of \cite{KR}.

Write $\Set^\cc_{\mono,\pbtxt}$ for the subcategory of $\left( \Set_{\mono} \right)^2$ whose objects are monomorphisms in $\Set^\cc$ and whose arrows are pullback squares consisting of monomorphisms. By Theorem \ref{characterisation-of-pure-effective-squares-if-llp} a square in $\Set^\cc_\pure$ is $\ind$-independent if and only if it is a pullback square. The following diagram is thus a pullback of categories:
\[
\xymatrix@=1.5pc{
  \Set^\cc_{\pure,\ind} \ar[r] \ar[d] & \Set^\cc_{\mono,\pbtxt} \ar[d]^G \\
  \left( \Set^\cc_\pure \right)^2 \ar[r]_F & \left( \Set^\cc_{\mono} \right)^2
}
\]
The inclusion functors $F$ and $G$ preserve directed colimits, the former because of Fact \ref{pure-mono-facts}(iv) and the latter because pullbacks commute with directed colimits in a locally finitely presentable category \cite[Proposition 1.59]{AR}. Furthermore, $\Set^\cc_{\mono,\pbtxt}$ is accessible \cite[Example 5.3]{LRV}, so $F$ and $G$ are accessible functors.

The subcategories exhibited by $F$ and $G$ are both closed under isomorphisms. So the above pullback is a pseudopullback \cite[Theorem 1]{JS}. As the $2$-category of accessible categories is closed under pseudopullbacks (\cite[Theorem 5.1.6]{MP} or \cite[Exercise 2.n]{AR}), we conclude that $\Set^\cc_{\pure,\ind}$ is accessible.
\end{proof}
The proof strategy for Theorem \ref{characterisation-of-pure-effective-squares-if-llp} will be as follows. We first introduce a technical condition in the definition below (which is due to \cite{Mu}). The point is that it provides a sufficient condition for a square to be pure effective (Lemma \ref{no-connection-outside-base-implies-pure-effective}). If $\cc$ is locally linearly preordered then this condition trivialises (Lemma \ref{llp-implies-no-connection-outside-base}).
\begin{defi}
\label{path-outside-base}
Let $K \subseteq L$ be an inclusion of presheaves, and let $a, b \in L \setminus K$. A \emph{path between $a$ and $b$ outside $K$} is a finite sequence $a = c_0, c_1, \ldots, c_n = b$ such that for all $0 \leq i < n$ we have $c_i \in \gen{c_{i+1}}$ or $c_{i+1} \in \gen{c_i}$. We call $n$ the \emph{length} of the path. If there is such a path we call $a$ and $b$ \emph{connected outside $K$}.

Being connected outside $K$ is an equivalence relation on $L \setminus K$ (possibly relating elements in different sorts). We write $C^L_K(a)$ for the equivalence class of $a$ under this relation. That is:
\[
C^L_K(a) = \{ b \in L \setminus K : \text{$a$ and $b$ are connected outside $K$} \}.
\]
For a subset $A \subseteq L \setminus K$ we write
\[
C^L_K(A) = \bigcup_{a \in A} C^L_K(a).
\]
\end{defi}
\begin{lemma}
\label{no-connection-outside-base-implies-pure-effective}
Any pullback square of pure monomorphisms of presheaves
\[
\xymatrix@=1pc{
  A \ar[r] & L \\
  K \ar[u] \ar[r] & B \ar[u]
}
\]
such that $C^L_K(A \setminus K) \cap C^L_K(B \setminus K) = \emptyset$, is pure effective.
\end{lemma}
\begin{proof}
Let $P$ be the pushout of $A \leftarrow K \to B$. One quickly checks that the induced arrow $P \to L$ is a monomorphism (presheaf categories have effective unions). So we may, and will, assume that all the arrows involved are inclusions. Let $\Sigma(x_1, \ldots, x_n)$ be a finite system of equations with parameters from $P = A \cup B$. Suppose that $c_1, \ldots, c_n \in L$ is a solution to $\Sigma$. We will show that $\Sigma$ has a solution in $P$ as well.

First, we partition $\Sigma$ into the following sets of equations:
\begin{itemize}
\item $\Sigma_1 = \{ (f \cdot x_i = g \cdot x_i) \in \Sigma \}$,
\item $\Sigma_2 = \{ (f \cdot x_i = g \cdot x_j) \in \Sigma : i \neq j \}$,
\item $\Sigma^A_3 = \{ (f \cdot x_i = a) \in \Sigma : a \in A \setminus K \}$,
\item $\Sigma^B_3 = \{ (f \cdot x_i = b) \in \Sigma : b \in B \setminus K \}$,
\item $\Sigma^K_3 = \{ (f \cdot x_i = d) \in \Sigma : d \in K \}$.
\end{itemize}
\begin{claim}
\label{proving-pure-effective-partition-of-variables}
There is a partition $(X_A, X_B)$ of the set of variables $\{x_1, \ldots, x_n\}$ of $\Sigma$ such that:
\begin{enumerate}[label=(\roman*)]
\item all the variables appearing in $\Sigma^A_3$ are contained in $X_A$,
\item all the variables appearing in $\Sigma^B_3$ are contained in $X_B$,
\item whenever $(f \cdot x_i = g \cdot x_j) \in \Sigma_2$ is such that $x_i \in X_A$ and $x_j \in X_B$ then $f \cdot c_i = g \cdot c_j \in K$.
\end{enumerate}
\end{claim}
\begin{proof}
Write $\Sigma' = \Sigma \setminus \{ (f \cdot x_i = g \cdot x_j) \in \Sigma : f \cdot c_i = g \cdot c_j \in K \}$. We define a graph $G$ with vertices $P \cup \{x_1, \ldots, x_n\}$, where two vertices have an edge if they both appear in an equation in $\Sigma'$. Let $G_c$ be the graph obtained from $G$ by replacing the vertex $x_i$ by $c_i$ for every $1 \leq i \leq n$, so the vertices of this new graph are all elements in $L$. We show that an edge between $v,v' \in G_c \setminus K$ represents a path between $v$ and $v'$ outside $K$. Since edges come from equations in $\Sigma'$, and ignoring the case $v = v'$, there are three cases.
\begin{itemize}
\item The edge comes from $\Sigma_2 \cap \Sigma'$, say from the equation $(f \cdot x_i = g \cdot x_j)$, so w.l.o.g.\ $v = c_i$ and $v' = c_j$. Set $d = f \cdot c_i = g \cdot c_j$. By construction of $\Sigma'$ we have $d \not \in K$, and so $c_i, d, c_j$ is a path outside $K$.
\item The edge comes from $\Sigma^A_3$, say from the equation $(f \cdot x_i = a)$, so $v, v'$, which are $c_i$ and $a$, is a path outside $K$.
\item The edge comes from $\Sigma^B_3$, say from the equation $(f \cdot x_i = b)$, so $v, v'$, which are $c_i$ and $b$, is a path outside $K$.
\end{itemize}
As $C^L_K(A \setminus K) \cap C^L_K(B \setminus K) = \emptyset$, there are no $a \in A \setminus K$ and $b \in B \setminus K$ in the same connected component. We can thus define
\[
X_A = \{ x_i : 1 \leq i \leq n \text{ and $x_i$ is connected in $G$ to some } a \in A \setminus K \}
\]
and
\[
X_B = \{x_1, \ldots, x_n\} \setminus X_A.
\]
By construction this is a partition of variables and by the above discussion it satisfies (i) and (ii).

For item (iii) we let  $(f \cdot x_i = g \cdot x_j) \in \Sigma_2$ is such that $x_i \in X_A$ and $x_j \in X_B$, and suppose for a contradiction that $f \cdot c_i = g \cdot c_j \not \in K$. Then $(f \cdot x_i = g \cdot x_j) \in \Sigma'$, and so there would be an edge in $G$ between $x_i$ and $x_j$. So $x_j$ is connected to $x_i$, which is connected to some $a \in A \setminus K$. Hence $x_j \in X_A$, contradicting $x_j \in X_B$. We thus conclude that $f \cdot c_i = g \cdot c_j \in K$.
\end{proof}
Equations in $\Sigma$ that have a variable in both $X_A$ and $X_B$ are all contained in $\Sigma_2$. Enumerate such equations as $\{ f_\ell \cdot x_{i_\ell} = g_\ell \cdot x_{j_\ell} : 1 \leq \ell \leq m \}$. Let $1 \leq \ell \leq m$, by symmetry of equations, we may assume that $x_{i_\ell} \in X_A$ and $x_{j_\ell} \in X_B$. Furthermore, writing $d_\ell = f_\ell \cdot c_{i_\ell} = g_\ell \cdot c_{j_\ell}$, we have that $d_\ell \in K$ by Claim \ref{proving-pure-effective-partition-of-variables}(iii).

Now define two new sets of equations:
\[
\Delta_A = \{ \varphi \in \Sigma : \text{the variables of $\varphi$ are contained in $X_A$} \} \cup \{ f_\ell \cdot x_{i_\ell} = d_\ell : 1 \leq \ell \leq m \}
\]
and
\[
\Delta_B = \{ \varphi \in \Sigma : \text{the variables of $\varphi$ are contained in $X_B$} \} \cup \{ g_\ell \cdot x_{j_\ell} = d_\ell : 1 \leq \ell \leq m \}.
\]
Then $\Delta_A$ is a finite system of equations with parameters from $A$ and with a solution in $L$ (restrict $(c_1, \ldots, c_n)$ to the subtuple corresponding to $X_A$). As $A \to L$ is pure, there is a solution $\bar{a} \in A$ to $\Delta_A$. Similarly, there is a solution $\bar{b} \in B$ to $\Delta_B$. We claim that $\bar{a} \cup \bar{b}$ is a solution to $\Sigma$ in $P$. Let $\varphi \in \Sigma$, we consider three cases.
\begin{itemize}
\item The variables of $\varphi$ are contained in $X_A$, so $\varphi$ will only care about elements in $\bar{a}$. By construction $\varphi \in \Delta_A$, and so $\bar{a}$ satisfies $\varphi$ in $A$, and hence in $P$.
\item The variables of $\varphi$ are contained in $X_B$. This case is analogous to the above.
\item There are two variables in $\varphi$, one of which is in $X_A$ and the other is in $X_B$. Then there is $1 \leq \ell \leq m$ such that $\varphi$ is the equation $f_\ell \cdot x_{i_\ell} = g_\ell \cdot x_{j_\ell}$. Let $a \in \bar{a}$ match $x_{i_\ell}$ and let $b \in \bar{b}$ match $x_{j_\ell}$. By construction of $\Delta_A$ and choice of $\bar{a}$ we have $f_\ell \cdot a = d_\ell$ in $A$ and hence in $P$. Similarly, we have $g_\ell \cdot b = d_\ell$ in $P$. So $f_\ell \cdot a = g_\ell \cdot b$ in $P$, as required.
\end{itemize}
We conclude that $\Sigma$ can be solved in $P$, and so $P \to L$ is pure.
\end{proof}
The converse to Lemma \ref{no-connection-outside-base-implies-pure-effective} does not hold, but we omit the example as it is not relevant for the remainder of this paper.
\begin{lemma}
\label{llp-implies-small-connecting-paths}
Suppose that $\cc$ is a locally linearly preordered category. Let $K \subseteq L$ be an inclusion of presheaves, and let $a, b \in L \setminus K$. Then $a$ and $b$ are connected outside $K$ if and only if there is a path between $a$ and $b$ outside $K$ of length at most $2$.

Furthermore, if $a = c_0, c_1, c_2 = b$ is a path outside $K$ of minimal length then $c_1 \in \gen{a} \cap \gen{b}$.
\end{lemma}
\begin{proof}
Let $a = c_0, c_1, \ldots, c_n = b$ be a path outside $K$ of minimal length. First, we claim that for every $0 < i < n$ we must have $c_i \in \gen{c_{i-1}}$ and $c_i \in \gen{c_{i+1}}$. Suppose for a contradiction that this is not the case. Without loss of generality, we may assume $c_i \not \in \gen{c_{i-1}}$, and so we must have $c_{i-1} \in \gen{c_i}$. We distinguish two cases.
\begin{itemize}
\item If $c_i \in \gen{c_{i+1}}$ then $c_{i-1} \in \gen{c_{i+1}}$, and so we could have omitted $c_i$ from the path, contradicting minimality.
\item If $c_{i+1} \in \gen{c_i}$ then there are $f,g \in \cc$ such that $c_{i-1} = f \cdot c_i$ and $c_{i+1} = g \cdot c_i$. As $\cc$ is locally linearly preordered, there must be $h \in \cc$ such that $f = hg$, in which case $c_{i-1} = h \cdot c_{i+1}$, or $g = hf$, in which case $c_{i+1} = h \cdot c_{i-1}$. In either case we see that we could have omitted $c_i$ from the path, again contradicting minimality.
\end{itemize}
Now that we have proved the claim, we suppose for a contradiction that $n > 2$. Then the claim applies to $i = 1$ and $i = 2$. The first application tells us that $c_1 \in \gen{c_0}$, while the second application tells us that $c_2 \in \gen{c_1}$, again arriving at a contradiction just like in the first bullet point above. The final sentence follows directly from the claim applied to the case $n = 2$.
\end{proof}
\begin{lemma}
\label{llp-implies-no-connection-outside-base}
Suppose that $\cc$ is a locally linearly preordered category. For any pullback square of monomorphisms of presheaves
\[
\xymatrix@=1pc{
  A \ar[r] & L \\
  K \ar[u] \ar[r] & B \ar[u]
}
\]
we have that $C^L_K(A \setminus K) \cap C^L_K(B \setminus K) = \emptyset$.
\end{lemma}
\begin{proof}
Towards a contradiction, suppose that there are $a \in A \setminus K$ and $b \in B \setminus K$ that are connected outside $K$. Let $n$ be minimal so that there is a path between $a$ and $b$ outside $K$ of length $n$. As the square is a pullback, we must have $n > 0$. By Lemma \ref{llp-implies-small-connecting-paths} we also must have $n \leq 2$. So we distinguish two cases.
\begin{itemize}
\item The case $n = 1$. Then either $a \in \gen{b}$ or $b \in \gen{a}$. Without loss of generality, assume the former. Then $a \in A \cap \gen{b} \subseteq A \cap B = K$, contradicting $a \not \in K$.
\item The case where $n = 2$. Let $a = c_0, c_1, c_2 = b$ be a path outside $K$ of minimal length. Then by the last sentence in Lemma \ref{llp-implies-small-connecting-paths} we have that $c_1 \in \gen{a} \cap \gen{b} \subseteq A \cap B = K$, contradicting that the path is outside $K$.
\end{itemize}
\end{proof}
\begin{repeated-theorem}[Theorem \ref{characterisation-of-pure-effective-squares-if-llp}]
If $\cc$ is a locally linearly preordered category then a commutative square in $\Set^\cc$ is pure effective if and only if it is a pullback square consisting of pure monomorphisms.
\end{repeated-theorem}
\begin{proof}
Every pushout square of monomorphisms is a pullback square. So pure effective squares are pullback squares because postcomposing a pullback square with a monomorphism results in a pullback square. The converse follows from Lemma \ref{no-connection-outside-base-implies-pure-effective}, where the Lemma \ref{llp-implies-no-connection-outside-base} takes care of the additional technical condition.
\end{proof}

\section{Classification for cofibrant generation}
\label{sec:classification-for-cofibrant-generation}
We start with the statement of the main theorem of the paper and present a short proof, which follows from the results obtained in the previous sections. Then we present several applications of the main theorem.
\begin{theo}
\label{cofibrantly-generated-classification}
The following are equivalent for a small category $\cc$.
\begin{enumerate}[label=(\roman*)]
\item The pure monomorphisms in $\Set^\cc$ are cofibrantly generated.
\item The pure effective squares form a stable independence relation on $\Set^\cc_\pure$.
\item The pullback squares form a stable independence relation on $\Set^\cc_\pure$.
\item There is a stable independence relation on $\Set^\cc_\pure$.
\item There is no presheaf on $\cc$ that interprets a complete bipartite graph.
\item There is no presheaf on $\cc$ that has the span-induced order property.
\item The category $\cc$ is locally linearly preordered.
\end{enumerate}
\end{theo}
\begin{proof}
The equivalence of (i), (ii) and (iv) is exactly Fact \ref{cofibrant-generation-iff-stable}. The implication (iii) $\Rightarrow$ (iv) is trivial, while (vii) implies both (ii) and (iii) at the same time by Theorem \ref{characterisation-of-pure-effective-squares-if-llp} and Corollary \ref{pure-effective-squares-are-stable-if-llp}. We can thus close the loop of implications with (ii) $\Rightarrow$ (v) $\Rightarrow$ (vi) $\Rightarrow$ (vii), which are Theorem \ref{stability-implies-no-complete-bipartite-graph}, Lemma \ref{order-property-implies-complete-bipartite-graph} and Theorem \ref{no-order-property-implies-llp}, respectively.
\end{proof}
Recall that an object $L$ is \emph{pure injective} if for any span $K' \xleftarrow{f} K \xrightarrow{g} L$ with $f$ a pure monomorphism, there is an arrow $h: K' \to L$ making the triangle commute: $hf = g$. A category is said to have \emph{enough pure injectives} if every object admits a pure monomorphism into a pure injective.
\begin{repeated-theorem}[Corollary \ref{enough-pure-injectives}]
If $\cc$ is a locally linearly preordered category then $\Set^\cc$ has enough pure injectives.
\end{repeated-theorem}
\begin{proof}
By Theorem \ref{cofibrantly-generated-classification} the pure monomorphisms are cofibrantly generated in $\Set^\cc$. They thus form the left part of a weak factorisation system \cite[Proposition 1.3]{B}. So for any presheaf $K$ the factorisation of $K \to 1$ gives a pure monomorphism into a pure injective.
\end{proof}
\begin{rem}
\label{strengthening-borceux-rosicky}
We recall the statement of \cite[Theorem 2.4]{BRo} (in our terminology): if in a locally finitely presentable category every pullback square of pure monomorphisms is pure effective then it has enough pure injectives. In fact, the proof of that theorem shows that the pure monomorphisms are cofibrantly generated, from which having enough pure injectives follows in the same way as in Corollary \ref{enough-pure-injectives}. The proof strategy of \cite{BRo} reminds us of ours: in Section \ref{sec:characterising-pure-effective-squares} we use a similar condition involving pullback squares to show that the pure effective squares form a stable independence relation, which in turn gives us cofibrant generation through Fact \ref{cofibrant-generation-iff-stable}. The condition of being locally linearly preordered even appears as part of the conditions of \cite[Corollary 2.6]{BRo}, which allows one to apply \cite[Theorem 2.4]{BRo} to certain presheaf categories. Our main improvements are that we removed the second condition from \cite[Corollary 2.6]{BRo}, and the fact that this fully characterises when cofibrant generation happens (i.e., a converse statement).
\end{rem}
\begin{exams}\label{some-examples}
We give some examples of applications of Theorem \ref{cofibrantly-generated-classification}. We start with existing examples, which quickly follow because being locally linearly preordered is an easy condition to verify.
\begin{enumerate}[label=(\arabic*)]
\item Let $\N^+ = (\N; 0, +)$ be the monoid of natural numbers with addition. Then $\N^+$ is locally linearly preordered and so the pure monomorphisms in $\Set^{\N^+}$ are cofibrantly generated. In particular, by Corollary \ref{enough-pure-injectives}, there are enough pure injectives, and so we recover \cite[Example 2.7(2)]{BRo} (which in turn cites \cite{W, T} for this result).
\item Any groupoid (i.e., a small category where every arrow is an isomorphism) is locally linearly preordered. So the pure monomorphisms in the category of presheaves on a groupoid are cofibrantly generated. This also appears as \cite[Corollary 2.8]{BRo} (which in turn generalises \cite[Proposition 3]{Ba}, see also \cite[Corollary 2.9]{BRo}). The statement there only mentions having enough pure injectives (cf.\ Corollary \ref{enough-pure-injectives}), but their proof establishes cofibrant generation (see also Remark \ref{strengthening-borceux-rosicky}).
\end{enumerate}
We also give some new examples.
\begin{enumerate}[label=(\arabic*)]
\setcounter{enumi}{2}
\item Let $P$ be a poset. Then $P$ can be viewed as a category as usual: the objects are the elements of $P$ and for $x,y \in P$ there is an arrow $x \to y$ precisely when $x \leq y$. Such a poset $P$ is locally linearly preordered if and only if for every $x \in P$ its upper set ${\uparrow} x = \{y \in P : x \leq y\}$ is a linear order, characterising when the pure monomorphisms in presheaves on posets are cofibrantly generated.
\item Let $\Delta$ be the category of finite linear orders and order preserving homomorphisms. Recall that a \emph{simplicial set} is a functor $\Delta^\textup{op} \to \Set$. The pure monomorphisms in the category of simplicial sets $\Set^{\Delta^\textup{op}}$ are not cofibrantly generated. For example, $\{0\} \xrightarrow{0 \mapsto 0} \{0 < 1\} \xleftarrow{0 \mapsto 1} \{0\}$ cannot be completed to a commuting triangle.
\end{enumerate}
\end{exams}
Surprisingly, it was unknown whether or not pure monomorphisms in the category of acts over $(\N; 1, \times)$ are cofibrantly generated. It now easily follows from Theorem \ref{cofibrantly-generated-classification} that this is not the case. Given its contrast to the additive monoid of natural numbers (see Example \ref{some-examples}(1)), we highlight it as a separate corollary.
\begin{coro}
\label{multiplicative-natural-numbers-not-cof-gen}
Let $\N^\times = (\N; 1, \times)$ be the monoid of natural numbers with multiplication. Then the pure monomorphisms in $\Set^{\N^\times}$ are not cofibrantly generated.
\end{coro}
\begin{proof}
For example, there is no $n \in \N$ such that $2 = n \times 3$ or $3 = n \times 2$, so $\N^\times$ is not locally linearly preordered.
\end{proof}
It is well known that having enough $\cm$-injectives for a class of arrows $\cm$ does not imply that $\cm$ is cofibrantly generated. For example, we can take $\cm$ to be the class of embeddings in the category of posets. By \cite[Propositions 1 and 2]{BB} the category of posets has enough $\cm$-injectives, but by \cite[Proposition 3.4]{AHRT} $\cm$ is not cofibrantly generated. However, this is not the same as our setup, so we leave the converse of Corollary \ref{enough-pure-injectives} as a question.
\begin{question}
\label{question-enough-injectives-but-not-cofibrantly-generated}
Is there a presheaf category with enough pure injectives such that the pure monomorphisms are not cofibrantly generated?
\end{question}


\begin{thebibliography}{ABCD}

\bibitem{AHRT} J. Ad\'{a}mek, H. Herrlich, J. Rosick\'{y} and W. Tholen, {\em Weak Factorization Systems and Topological Functors}, Applied Categorical Structures 10 (2002), 237-249.

\bibitem{AR} J. Ad\'{a}mek and J. Rosick\'{y}, {\em Locally Presentable and Accessible Categories}, Cambridge University Press (1994).

\bibitem{AR-pure} J. Ad\'{a}mek and J. Rosick\'{y}, {\em On Pure Quotients and Pure Subobjects}, Czechoslovak Mathematical Journal 54 (2004), 623-636.

\bibitem{Ba} B. Banaschewski, {\em Equational Compactness of G-Sets}, Canadian Mathematical Bulletin 17 (1974), 11-18.

\bibitem{BB} B. Banaschewski and G. Bruns, {\em Categorical characterization of the MacNeille completion}, Archiv der Mathematik 18 (1967), 369-377.

\bibitem{BN} B. Banaschewski and E.~M. Nelson, {\em Equational compactness in equational classes of algebras}, Algebra Universalis 2 (1972), 152-165;

\bibitem{Barr} M. Barr, {\em On categories with effective unions}, Categorical Algebra and its Applications, Lecture Notes in Mathematics 1348, Springer-Verlag (1988), 19--35.

\bibitem{B} T. Beke, {\em Sheafifiable homotopy model categories}, Mathematical Proceedings of the Cambridge Philosophical Society 129 (2000), 447-475.

\bibitem{BRo} F. Borceux and J. Rosick\'y, {\em Purity in algebra}, Algebra Universalis 56 (2007), 17-35.

\bibitem{JS} A. Joyal and R. Street, {\em Pullbacks equivalent to pseudopullbacks}, Cahiers de Topologie et G\'{e}om\'{e}trie Diff\'{e}rentielle Cat\'{e}goriques 34 (1993), 153-156.

\bibitem{K} M. Kamsma, {\em NSOP$_1$-like independence in AECats}, The Journal of Symbolic Logic 89 (2024), 724-757.

\bibitem{KR} M. Kamsma and J. Rosick\'{y}, {\em Lifting independence along functors}, Applied Categorical Structures 34 (2025).

\bibitem{KKM} M. Kilp, U. Knauer, A. V. Mikhalev, {\em Monoids, Acts and Categories}, De Gruyter (2011).

\bibitem{LPRV} M. Lieberman, L. Positselski, J. Rosick\'y and S. Vasey, {\em Cofibrant generation of pure monomorphisms}, Journal of Algebra 560 (2020), 1297-1310.

\bibitem{LRV} M. Lieberman, J. Rosick\'y and S. Vasey, {\em Forking independence from the categorical point of view}, Advances in Mathematics 346 (2019), 719-772.

\bibitem{LRV1} M. Lieberman, J. Rosick\'y and S. Vasey, {\em Cellular categories and stable independence}, The Journal of Symbolic Logic 88 (2023).

\bibitem{M-AR} M. Mazari-Armida and J. Rosick\'{y}, \emph{Relative injective modules, superstability and noetherian categories}, Journal of Mathematical Logic, online (2024).

\bibitem{Mu} T. G. Mustafin, {\em Stability of the theory of polygons}, Proceedings of the Institute of Mathematics 8 (1988), 92-108 (in Russian); translated in Model Theory and Applications, American Mathematical Society translations 295 (1999), 205-223.

\bibitem{MP} M. Makkai and R. Par\'e, {\em Accessible Categories: The Foundation of Categorical Model Theory}, Contemporary Mathematics 104 (1989).

\bibitem{P} M. Prest, {\em Purity, Spectra and Localisation}, Cambridge University Press (2009).

\bibitem{Pr} H. Pr\"ufer, {\em Untersuchungen \"uber die Zerlegbarkeit der abz\"ahlbaren prim\"aren Abelschen Gruppen}, Mathematische Zeitschrift 17 (1923), 35–61.

\bibitem{R} P. Rothmaler, {\em Purity in model theory}, in Advances in algebra and model theory (ed.\ by M. Droste and R. G\"obel), Gordon and Breach 1997, 445–470.

\bibitem{ST}  J. \v Saroch and J. Trlifaj, {\em Test sets for factorization properties of modules}, Rendiconti del Seminario Matematico della Universit\`a di Padova 144 (2020), 217-238.

\bibitem{Shelah} S. Shelah, {\em Classification theory and the number of nonisomorphic models}, North-Holland Publishing (1990).

\bibitem{T2} W. Taylor, {\em Residually small varieties}, Algebra Universalis 2 (1972), 33-53.

\bibitem{T} W. Taylor, {\em Pure-irreducible mono-unary algebras}, Algebra Universalis (1974), 235-243.

\bibitem{T3} W. Taylor, {\em Pure compactifications in quasi-primal varieties}, Canadian Journal of Mathematics 28 (1976), 50-62.
 
\bibitem{TZ} K. Tent and M. Ziegler, {\em A Course in Model Theory}, Cambridge University Press (2012).

\bibitem{W2} G.H. Wenzel {\em Equational compactness}, in: G. Gr\"atzer, Universal Algebra, Springer-Verlag (1979).

\bibitem{W} G.H. Wenzel, {\em Subdirect irreducibility and equational compactness in unary algebras}, Archiv der Mathematik 21 (1970), 256-264.
 
\end{thebibliography}
\end{document}